\documentclass{article}
\usepackage{a4wide}

\usepackage[english]{babel}
\usepackage[utf8]{inputenc}

\usepackage{amsmath,amsfonts,amssymb,amscd,amsthm}
\usepackage[linktocpage,unicode]{hyperref}

\usepackage{latexsym}
\usepackage{enumitem}

\hypersetup{
    pdftoolbar=true,        
    pdfmenubar=true,        
    pdffitwindow=false,     
    pdfstartview={FitH},    
    pdfnewwindow=true,      
    colorlinks=true,       
    linkcolor=red,          
    citecolor=blue,        
    filecolor=magenta,      
    urlcolor=cyan,           
    unicode=true          
}

\newtheorem{theo}{Theorem}[section]
\newtheorem{coro}[theo]{Corollary}
\newtheorem{lemm}[theo]{Lemma}
\newtheorem{prop}[theo]{Proposition}

\theoremstyle{definition}
\newtheorem{defi}[theo]{Definition}
\newtheorem{rema}[theo]{Remark}

\DeclareMathOperator{\run}{run}
\newcommand{\pp}{^+}
\newcommand{\mm}{^-}
\def\Sym{{\bf Sym}}

\DeclareMathOperator{\pk}{pk}
\DeclareMathOperator{\des}{des}

\DeclareMathOperator{\Des}{Des}
\DeclareMathOperator{\Vect}{Vect}

\author{
Matthieu Josuat-Vergès\thanks{Supported by the Laboratoire International Franco-Québecois de Recherche en Combinatoire (LIRCO), and by ANR CARMA (ANR-12-BS01-0017).} \\
\small CNRS and Institut Gaspard Monge, Université Paris-Est Marne-la-Vallée\\
\small 5 Boulevard Descartes, Champs-sur-Marne, 77454 Marne-la-Vallée cedex 2, FRANCE \\
\small\tt matthieu.josuat-verges@u-pem.fr \\
\and
C.Y. Amy Pang\thanks{Supported by a CRM-ISM postdoctoral fellowship, and an NSERC individual research grant of Fran\c cois Bergeron.} \\
\small Laboratoire de Combinatoire et d’Informatique Mathématique, Université du Québec à Montréal \\  
\small CP 8888, Succ. Centre-ville, Montréal, Québec, H3C 3P8, CANADA \\
\small\tt amypang@lacim.ca \\
\small Department of Mathematics, Hong Kong Baptist University\\ 
\small Kowloon Tong, Kowloon, HONG KONG \\
\small\tt amypang@hkbu.edu.hk
}

\title{Subalgebras of Solomon's descent algebra based on alternating runs}

\begin{document}

\maketitle

\begin{abstract}
The number of alternating runs is a natural permutation statistic. We show it can 
be used to define some commutative subalgebras of the symmetric group algebra, and more 
precisely of the descent algebra. 
The Eulerian peak algebras naturally appear as subalgebras of our run algebras. 
We also calculate the orthogonal idempotents for run algebras in terms of noncommutative
symmetric functions.
\end{abstract}

\section{Introduction}

An {\it alternating run} of a permutation $\sigma\in\mathfrak{S}_n$ is a maximal monotone sequence 
of consecutive elements in $\sigma_1, \dots, \sigma_n $.
For example, the alternating runs of $\sigma=14523687$ are $145$, $52$, $2368$, and $87$.
We denote by $\run(\sigma)$ the number of alternating runs of $\sigma$.
The enumeration of permutations refined by the number of alternating runs have 
been the subject of various previous works, beginning with André \cite{andre}
as early as the late 19th century, then Carlitz in the 70's \cite{carlitz,carlitz1,carlitz2,carlitz3}.
See also \cite{bonaehrenborg,canfieldwilf,chowma,fewster,ma1,ma2,zhuang} for more recent references.
It is worth mentioning that the work \cite{fewster} was motivated by computations in quantum field
theory. Also, alternating runs play a role in some algorithms such as pattern matching \cite{bruner}. 

 However, we take here a different perspective, our goal being to define and study various algebras based on this notion of alternating runs. Such an algebra first arose in Doyle and Rockmore's study of ``ruffle" card-shuffling: \cite[Section~5.5]{doyle} shows that the elements
\[
  W_k =  \sum_{ \substack{  \sigma \in \mathfrak{S}_n \\ \run(\sigma)=k  }   } \sigma,
\]
for $1\leq k \leq n-1$, linearly span a commutative subalgebra of $\mathbb{Z}[\mathfrak{S}_n]$, 
called the {\it reduced turning algebra}.
Doyle and Rockmore also considered the number of alternating runs 
when we add an initial $0$ in front of the permutation and obtain in this way another algebra, 
called the {\it turning algebra} \cite[Section~5.4]{doyle}. We will in particular obtain new proofs 
of these results in the present article.

In fact, it is easily seen that $\run(\sigma)$ only depends on the descent set of $\sigma$,
so that the elements $W_k$ lie in the descent algebra $\mathcal{D}_n \subseteq \mathbb{Z}[\mathfrak{S}_n]$
(see \cite{solomon} or the next section for a definition).
This widely studied algebra provides other examples of combinatorial statistics such 
that there is an algebra linearly spanned by sums of permutations having the same value.
Consider for example the descent statistic, then the $n$ elements 
\[
  E_k = \sum_{ \substack{  \sigma \in \mathfrak{S}_n \\ \des(\sigma)=k  }   } \sigma, 
\]
for $0\leq k \leq n-1$ linearly span a subalgebra of $\mathcal{D}_n$ called the {\it Eulerian algebra} $\mathcal{E}_n$
(see~\cite{gelfand,loday}).
But those being particularly relevant here are the Eulerian peak algebra and its left peak analogue 
\cite{aguiarnymanorellana,novellisaliolathibon,petersen,schocker}, because peaks and alternating runs are tightly connected.

Our main goal is to introduce a new algebra based on alternating runs, as the linear span of the $2n-2$ elements:
\begin{align*}
  W_k^+   := \sum_{ \substack{  \sigma \in \mathfrak{S}_n \, \run(\sigma)=k, \\ \text{ the first run is ascending  }  }} \sigma  , \qquad
  W_k^-   := \sum_{ \substack{  \sigma \in \mathfrak{S}_n \, \run(\sigma)=k, \\ \text{ the first run is descending }  }} \sigma,
\end{align*}
for $1\leq k \leq n-1$. This algebra is noncommutative, but it contains several commutative subalgebras, in particular 
the two considered by Doyle and Rockmore and  the two Eulerian peak subalgebras. 

We give in Section~\ref{sec:bij} a bijective proof of its existence. Then we show that once we know the existence 
of this $2n-2$ dimensional algebra, the existence of the various subalgebras easily follows.

Then in Section~\ref{sec:algproof}, we  provide a Hopf algebraic proof of the existence of run algebras, using the formalism 
of noncommutative symmetric functions \cite{gelfand}. Indeed, whereas each homogeneous component has an internal product 
which is essentially that of the descent algebra, the external Hopf structure permits to build new bases. 
Also, this algebraic approach easily gives the commutativity of the subalgebras, a property that seems very difficult 
to get by bijective proofs.

We go on in Section~\ref{sec:idem} by computing primitive orthogonal idempotents for each of the various run algebras, as generating functions in noncommutative
symmetric functions. Part of these calculations appear to be equivalent to Petersen's computation of peak idempotents via order polynomials of enriched $P$-partitions \cite{petersen}. 
We also give some properties of the related characters of the symmetric groups, in the same way as idempotents of the Eulerian subalgebra
are related with Foulkes characters \cite{foulkes}.

\subsection*{Acknowledgements}
We thank Jean-Yves Thibon for suggesting an exploration of algebras based on alternating runs. We thank Kyle Petersen for helpful discussions, and 
for pointing out to us the reference by Doyle and Rockmore. We are also grateful to Fran\c cois Bergeron, Christophe Hohlweg and Franco Saliola for their help and comments. 
Sage computer software \cite{sage} was very useful throughout this project.
We also thank the Laboratoire International Franco-Québécois de Recherche en Combinatoire (LIRCO), who supported the French author in his visits to Montréal.
Eventually, we thank the reviewers for their nice comments and suggestions.

\section{Preliminaries}

\subsection{Notations}

A permutation $\sigma\in\mathfrak{S}_n$ is denoted as the sequence $\sigma_1 , \dots , \sigma_n$ where $\sigma_i=\sigma(i)$.
We denote by $\Des(\sigma)$ the set of {\it descents} of $\sigma$, i.e.~integers $i$ such that $1\leq i \leq n-1$ and $\sigma(i)>\sigma(i+1)$.
Let also $\des(\sigma) = \# \Des(\sigma)$.
Let $[n-1]$ denote $\{1,\dots,n-1\}$. The {\it descent class} $\mathfrak{D}_I$ for $I\subseteq [n-1]$
is the set of $\sigma\in\mathfrak{S}_n$ with $\Des(\sigma)=I$.

A {\it partition} $\lambda$ of $n$ is a decreasing sequence of positive integers
$\lambda_1 \geq \lambda_2 \geq \dots \geq \lambda_k$ called {\it parts}, such that $\sum_{i=1}^k \lambda_i = n$. 
We also denote by $|\lambda|$ the sum of parts of $\lambda$.
The integer $k$ is the {\it length} of $\lambda$, and is denoted $\ell(\lambda)$. 
Also, we denote $\ell_e(\lambda)$ (respectively, $\ell_o(\lambda)$) the even length (respectively, the odd length) of $\lambda$,
i.e.~the number of even parts (respectively, odd parts). Write $z_{\lambda}$ for $1^{i_1}i_1!2^{i_2}i_2!\dots$, where $i_r$ is the number of parts of size $r$ in $\lambda$.
We denote $\lambda\vdash n$ to say that $\lambda$ is a partition of $n$.

A {\it composition} of $n$ is a sequence $I=(i_1,\dots,i_r)$ of positive integers such that $\sum_{j=1}^ r i_j = n$. 
We also denote by $|I|$ the sum of parts of $I$. 
The integer $r$ is called the {\it length} of $I$, and is denoted $\ell(I)$.
We write $I \vDash n$ to say that $I$ is a composition of $n$.

In a partition or a composition, we write $1^n$ to mean that 1 is repeated $n$ times.

The {\it descent set} of a composition $I=(i_1,\dots,i_r)$, denoted $\Des(I)$, is 
\[
 \Bigg\{ i_1, i_1+i_2,\dots, \sum_{j=1}^{r-1} i_j \Bigg\}. 
\]
The map $ I \mapsto \Des(I)$ is a bijection from compositions of $n$ to subsets of $[n-1]$.
When $I,J \vDash n$, we say that $J$ is a refinement of $I$, denoted $I \preccurlyeq J$, 
if $\Des(I)\subseteq \Des(J)$. When we write $I\preccurlyeq J$, we always understand that $|I|=|J|$.



\subsection{The descent algebra}

For $I\vDash n$, consider the following sums in the group algebra $\mathbb{Z}[\mathfrak{S}_n]$:
\[
  R_I = \sum_{\substack{ \sigma\in\mathfrak{S}_n \\  \Des(\sigma) = \Des (I) }} \sigma, \qquad
  S^I = \sum_{\substack{ \sigma\in\mathfrak{S}_n \\  \Des(\sigma) \subseteq \Des(I) }} \sigma.
\]
Note that by an immediate inclusion-exclusion, we have:
\begin{equation*} 
  S^I = \sum_{J \preccurlyeq I} R_J,  \qquad R_I = \sum_{J \preccurlyeq I} (-1)^{\ell(J)-\ell(I)} S^J.
\end{equation*}

The linear span of the elements $R_I$ (or equivalently, $S^I$) for $I\vDash n$ is denoted $\mathcal{D}_n$
and called the {\it descent algebra}. It was indeed shown by Solomon \cite{solomon} that $\mathcal{D}_n$ is a subalgebra
of the group algebra $\mathbb{Z}[\mathfrak{S}_n]$. However, to facilitate our construction of idempotents of various subalgebras, it will be more convenient to work with rational coefficients, i.e.~to consider  $\mathcal{D}_n$ as a subalgebra of $\mathbb{Q}[\mathfrak{S}_n]$.

\subsection{The Eulerian peak algebras.}
\label{peakseulerian}

Two previously identified subalgebras of the descent algebra will play a key role in the analysis of our new subalgebras. These are based on peaks:

\begin{defi}
 A {\it peak} of a permutation $\sigma\in\mathfrak{S}_n$ is an integer $i$ with $2\leq i \leq n-1$ and
 $\sigma(i-1)<\sigma(i)>\sigma(i+1)$. We denote by $\pk(\sigma)$ the number of peaks of $\sigma$.
\end{defi}

\begin{defi}
 A {\it left peak} of a permutation $\sigma\in\mathfrak{S}_n$ is an integer $i$ with $1\leq i \leq n-1$ 
 and $\sigma(i-1)<\sigma(i)>\sigma(i+1)$ (with the convention $\sigma(0)=0$). We denote by $\pk^{\circ}(\sigma)$ the 
 number of left peaks of $\sigma$.
\end{defi}

\begin{prop} \cite[Sec. 9]{schocker}
The elements
\[
  P_k = \sum_{\substack{ \sigma \in\mathfrak{S}_n \\ \pk(\sigma)=k  }} \sigma  
\]
for $0\leq k \leq \lfloor \frac{n-1}{2} \rfloor $ span the \emph{Eulerian peak algebra} $\mathcal{P}_n$, a commutative subalgebra of $\mathcal{D}_n$.
Note that $\dim(\mathcal{P}_n) = \lceil \frac{n}{2} \rceil$.
\end{prop}

The Eulerian peak algebra first appeared in Schocker's analysis \cite{schocker} of a larger, noncommutative, peak algebra, spanned by sums of permutations with the same peak positions. This work determined the idempotents of $\mathcal{P}_n$, which we will recover in the present Theorem~\ref{theoidemp}.

Aguiar, Bergeron and Nyman \cite{aguiarbergeronnyman} also studied this larger peak algebra, by viewing it as an ideal within the image of projecting the type B descent algebra to $\mathcal{D}_n$. This viewpoint uncovered a left peak variant of $\mathcal{P}_n$:

\begin{prop} \cite[Sec. 6.3]{aguiarbergeronnyman}
The elements
\[
  P^\circ_k = \sum_{\substack{ \sigma \in\mathfrak{S}_n \\ \pk^\circ(\sigma)=k  }} \sigma  
\]
for $0\leq k \leq \lfloor \frac{n}{2} \rfloor $ span the \emph{Eulerian left peak algebra} $\mathcal{P}_n^\circ$,  a commutative subalgebra of $\mathcal{D}_n$.
 Note that $\dim(\mathcal{P}_n^\circ) = \lfloor \frac{n}{2} \rfloor + 1$.
\end{prop}

Petersen \cite{petersen} computed a complete set of primitive orthogonal idempotents of both $\mathcal{P}_n$ and $\mathcal{P}_n^\circ$, in terms of enriched $P$-partitions. Our re-computation in the present Theorem~\ref{theoidemp} is uncannily similar, despite being in the different language of noncommutative symmetric functions.

Both $\mathcal{P}_n$ and $\mathcal{P}_n^\circ$ are subalgebras of our new algebra $\mathcal{W}^\pm_n$.

This idea of adding an initial 0 is also relevant in the case of runs. As mentioned in the introduction, this definition was used by Doyle and Rockmore~\cite{doyle}.

\begin{defi} \label{leftrun}
 With the convention $\sigma_0=0$, a {\it left run} of a permutation $\sigma\in\mathfrak{S}_n$ is a maximal monotone sequence of consecutive elements
 in $\sigma_0, \dots , \sigma_n$.
 We denote by $\run^\circ(\sigma)$ the number of left runs of $\sigma$.
\end{defi}

Note that if the first run of $\sigma$ is ascending, then $\run^\circ(\sigma) = \run(\sigma)$, and the first left run 
is obtained from the first run by adding the initial zero.
If the first run is descending, then $\run^\circ(\sigma) = \run(\sigma)+1$, and the first left run is $0,\sigma_1$.

%

\subsection{Noncommutative symmetric functions}

\label{nsym}

To find and express the idempotents of the run algebras, it is useful
to view the descent algebra from another perspective, that contains
extra algebraic structure.

Let $\Sym$ denote the Hopf algebra of noncommutative symmetric functions, which is the direct sum of $\Sym_n$, its degree $n$ homogeneous components (see \cite{gelfand} as a general reference on this subject).  
Its product $\star$ and coproduct $\Delta$ are defined below.
It will be convenient to also work in $\widehat{\Sym}$, the completion, which
allows infinite sums of the elements of $\Sym$, of unbounded degree. 

As a vector space, we have the identification $\Sym_{n}=\mathcal{D}_{n}$.
Under this identification, the image of $S^{I}$ are the {\it complete noncommutative symmetric functions},
and the image of $R_{I}$ are the {\it ribbon noncommutative symmetric functions}.
We write $S^{I},R_{I}$ for these images also.

The product on $\mathcal{D}_{n}$ is extended to $\Sym$ by the rule $xy=0$ if $x\in\Sym_m$, $y\in\Sym_n$, and $m\neq n$.
We call it the {\it internal product} of $\Sym$, to distinguish from the external product.
However, it is to be noted that the usual internal product of noncommutative symmetric functions is taken in the
opposite order (see \cite{gelfand}).

$\Sym$ also admits an {\it external product} $\star$, for which degree is additive:
$\Sym_{n}\star\Sym_{n'}\subseteq\Sym_{n+n'}$.
In the complete basis, this is given by concatenation of compositions:
\[
S^{(i_{1},\dots,i_{\ell(I)})}\star S^{(j_{1},\dots,j_{\ell(J)})}=S^{(i_{1},\dots,i_{\ell(I)},j_{1},\dots,j_{\ell(J)})}.
\]
Hence, for any composition $I=(i_1,\dots,i_r)$, we have $S^{I}=S_{i_{1}}\star\dots\star S_{i_{r}}$,
so $\Sym$ is a free algebra for the external product, generated by
the $S_{i}$, the complete noncommutative symmetric functions indexed
by compositions with a single part. 

In the ribbon basis, the external product results in ``near-concatenation'':
\[
R_{(i_{1},\dots,i_{\ell(I)})}\star R_{(j_{1},\dots,j_{\ell(J)})}=R_{(i_{1},\dots,i_{\ell(I)},j_{1},\dots,j_{\ell(J)})}+R_{(i_{1},\dots,i_{\ell(I)-1},i_{\ell(I)}+j_{1},j_{2},\dots,j_{\ell(J)}).}
\]
For example, $S^{(2,3)}\star S^{(5,1,1)}=S^{(2,3,5,1,1)}$ and $R_{(2,3)}\star R_{(5,1,1)}=R_{(2,3,5,1,1)}+R_{(2,8,1,1)}.$
We write $F^{\star k}$ for powers under the external product, to distinguish
from the internal product power $F^{k}$.

The two products are related by the {\it splitting formula} of \cite[Prop. 5.2]{gelfand};
this involves a third algebraic operation, the coproduct. This is
the (external) algebra homomorphism $\Delta:\Sym\rightarrow\Sym\otimes\Sym$
with $\Delta(S_{n})=\sum_{i=0}^{n}S_{i}\otimes S_{n-i}$. (The convention
is that $S_{0}=1$.) It follows that $\sum_{n\geq0}R_{n}$ and $\sum_{n\geq0}R_{(1^{n})}$
are {\it grouplike} --- that is, they satisfy $\Delta(G)=G\otimes G$.
It is a standard fact of Hopf algebras that (external) products of
grouplike elements are also grouplike. For grouplike $G,$ the splitting
formula simplifies to: 
\[
 G(F_{1}\star\dots\star F_{k})=(GF_{1})\star\dots\star(GF_{k}).
\]
This simplified version is the only case of the splitting formula
that we will need.

\begin{rema}
Since we are interested in subalgebras of the descent algebra, it is natural to use a special 
symbol for the external product of (noncommutative) symmetric functions, the product 
without a special symbol meaning the internal product. However this differs from the usual notation.
\end{rema}

\subsection{Symmetric functions}
\label{sym}
Let $Sym$ denote the algebra of symmetric functions and $Sym_n$ its degree $n$ homogeneous component.
We follow the notation of \cite{macdonald} so that $h_\lambda$, $e_\lambda$, $p_\lambda$, $s_\lambda$
are respectively the homogeneous, elementary, power sum, and Schur symmetric functions associated with 
a partition $\lambda$. If $I=(i_1,\dots,i_{\ell(I)})$ is a composition, there is a unique ribbon (i.e.~a skew
shape which is connected and contains no $2\times 2$ square) whose row lengths from top to bottom
in French notation are $i_1,\dots,i_{\ell(I)}$, and the skew Schur function associated with it is denoted $r_I$.
It is given by $r_I= \sum_{J\preccurlyeq I} (-1)^{\ell(I)-\ell(J)} h_J $ where $h_J = \prod h_j $ where $j$ runs through the parts of $J$.
To be coherent with our notation on noncommutative symmetric functions, the usual product of $Sym$
is denoted $\star$, whereas the {\it internal product} on $Sym_n$ has no particular symbol.



There is a well-known map $\Gamma:\Sym \rightarrow Sym$, defined by the two equivalent formulas:
\[
   \Gamma( R_I ) = r_{I}, \qquad \Gamma(S_I) = h_I.
\]
In particular, we have:
\[
  \Gamma(S_n) = h_n, \qquad \Gamma( R_{(1^n)} ) = e_n. 
\]
This map is an algebra morphism for both the internal and external product of symmetric functions:
\[
   \Gamma( x y ) = \Gamma(x)\Gamma(y),\qquad \Gamma ( x \star y ) = \Gamma(x) \star \Gamma(y).
\]

One fact about $Sym_n$  that will be particularly useful for us is that the rescaled power sums $\frac{1}{z_\lambda} p_\lambda$ form a basis of orthogonal idempotents (under the internal product). It follows that, if $\Lambda_1,\Lambda_2,\dots,\Lambda_K$  are disjoint subsets of partitions of $n$, then $\sum_{\lambda \in \Lambda_k} \frac{1}{z_\lambda} p_\lambda$ form a coarser set of orthogonal idempotents. In fact, every subalgebra of $Sym_n$ has a basis of orthogonal idempotents of this form, and identifying the $\Lambda_k$ is particularly interesting if we view $Sym_n$ as the character ring of $\mathfrak{S}_n$ \cite[Sec. 7.18]{stanley}. (In this light, $\Gamma$ is the homomorphism of Solomon from the descent algebra of a finite Coxeter group to its character ring \cite{solomon}.) For example, the idempotents of $\Gamma(\mathcal{E}_n)$, the symmetric function image of the Eulerian algebra, are $\sum_{\ell(\lambda)=k} \frac{1}{z_\lambda} p_\lambda$ \cite[Sec. 5.3]{gelfand}, corresponding to the indicator functions on permutations with $k$ cycles. By definition, $\sum_{\ell(I)=k} r_I$ is also a basis of $\Gamma(\mathcal{E}_n)$. So, writing $\chi_I$ for the character corresponding to $r_I$, this slickly recovers a result of Foulkes \cite[Th. 3.1]{foulkes}, \cite{kerberthurlings}:
\[
\left\{ \sum _{\ell(I)=k} \chi_I : 1\leq k \leq n \right\}
\]
gives a basis for the subspace of functions $\mathfrak{S}_n \rightarrow \mathbb{R}$ depending only on the number of cycles. 
Applying this idea to our new commutative run algebras yields similar characters that depend only on the numbers of odd cycles and of even cycles, see Corollary~\ref{foulkescharacters}.

\section{Bijective proof for the existence of run algebras}
\label{sec:bij}

Recall that $\run(\sigma)$ is defined at the beginning of the introduction, and $\run^\circ(\sigma)$ in Definition~\ref{leftrun}.

\begin{defi}
Let
\begin{align*}
 \mathfrak{W}_k\pp & := \{ \sigma \in \mathfrak{S}_n \,:\, \run(\sigma) = k \text{ and the first run is ascending} \}, \\
 \mathfrak{W}_k\mm & := \{ \sigma \in \mathfrak{S}_n \,:\, \run(\sigma) = k \text{ and the first run is descending} \},  \\
 \mathfrak{W}_k & := \{ \sigma \in \mathfrak{S}_n \,:\, \run(\sigma) = k \}  = \mathfrak{W}_k^+ \cup \mathfrak{W}_k^-,  \\
\end{align*}
for $1\leq k \leq n-1$, and 
\begin{align*}
 \mathfrak{W}^\circ_k & := \{ \sigma \in \mathfrak{S}_n \,:\, \run^\circ(\sigma) = k \}  = \mathfrak{W}_k^+ \cup \mathfrak{W}_{k-1}^-
\end{align*}
for $1\leq k \leq n$. The given range for $k$ indicates when these sets are nonempty. 
Note that in the definition of $\mathfrak{W}^\circ_k$ we need to specify that $\mathfrak{W}_n^+ = \mathfrak{W}_0^- = \emptyset$.
And let 
\begin{align*}
  W_k^-     := \sum_{\sigma \in \mathfrak{W}_k^- } \sigma, \qquad
  W_k^+     := \sum_{\sigma \in \mathfrak{W}_k^+ } \sigma, \qquad
  W_k       := \sum_{\sigma \in \mathfrak{W}_k } \sigma, \qquad
  W_k^\circ := \sum_{\sigma \in \mathfrak{W}_k^\circ } \sigma.
\end{align*}
Note that $W_k =  W_k^-  +  W_k^+$ and $W_k^\circ =  W_{k-1}^-  +  W_{k}^+$.
Then we define subspaces of the descent algebra $\mathcal{D}_n$ as follows:
\begin{align*}
 \mathcal{W}^{\pm} & := \Vect( W_k\pp,W_k\mm \, : \, 1\leq k \leq n-1 \}, \\
 \mathcal{W}       & := \Vect( W_k           \, : \, 1\leq k \leq n-1 \}, \\
 \mathcal{W}^\circ & := \Vect( W^\circ_k \, : \, 1\leq k \leq n \},          \\
 \mathcal{C}       & := \Vect( W_{2k} \, : \, 1\leq k \leq \lfloor\tfrac{n-1}{2}\rfloor \} \oplus
                        \Vect( W_{2k-1}\pp,W_{2k-1}\mm \, : \, 1\leq k \leq \lfloor\tfrac{n}{2}\rfloor \}.
\end{align*}
\end{defi}

Note that in each definition, the given generators are linearly independent, since they are sums of permutations over
disjoint nonempty subsets. It follows that:
\[
  \dim(\mathcal{W}^{\pm}) = 2n-2, \quad \dim(\mathcal{W}) = n-1, \quad \dim( \mathcal{W}^\circ) = n, \quad
  \dim(\mathcal{C}) = \lfloor \tfrac{3n-2}{2}  \rfloor.
\]

\begin{theo}  \label{subalg}
The space $\mathcal{W}^{\pm}$ is a subalgebra of $\mathcal{D}_n$, 
and the spaces $\mathcal{W}$, $\mathcal{W}^\circ$, $\mathcal{C}$, $\mathcal{P}$ and $\mathcal{P}^\circ$ are subalgebras of $\mathcal{W}^\pm$.
\end{theo}

Note that $\mathcal{W}$ and $\mathcal{W}^\circ$ are respectively the reduced turning algebra and the turning algebra of Doyle and Rockmore~\cite{doyle}.
And $\mathcal{P}$ and $\mathcal{P}^\circ$ are the Eulerian peak and Eulerian left peak algebra, as described in Subsection~\ref{peakseulerian}.

The rest of this section contains the bijective proof. In Subsections~\ref{subsec:connected} 
to~\ref{subsec:bijphi}, we prove the result for $\mathcal{W}^\pm$. 
Then, in Subsection~\ref{subsec:othersubalg}, we show how the other 
parts of the theorem follow.

The beginning of the proof, in particular using the connectedness in the left weak order, essentially follows 
Atkinson's proof of the existence of the descent algebra \cite{atkinson} (which is also in Tits \cite[\S6]{tits}
although in a somewhat different language and very concise style).

\subsection{Alternating runs and the left weak order}

\label{subsec:connected}

This subsection contains a  preliminary result about the sets $\mathfrak{W}_j^+$ and $\mathfrak{W}_j^-$.
Let $s_i$ denote the simple transposition $(i,i+1)\in\mathfrak{S}_n$.
Let us recall that the left weak order on $\mathfrak{S}_n$ is defined by the covering 
relations $\sigma \lessdot s_i \sigma $ if 
$i$ is to the left of $i+1$ in $\sigma$, for all $\sigma\in\mathfrak{S}_n$ and $1\leq i\leq n-1$.
It is known that descent classes are intervals for this order \cite{bjorner}.
In particular, each descent class is a connected subset in $\mathfrak{S}_n$ (seen as an undirected
graph whose edges are given by the cover relation of the left weak order).

\begin{lemm}
 Let $I,J\subseteq [n-1]$ with $I\neq J$. The descent classes $\mathfrak{D}_I$ and $\mathfrak{D}_J$ are 
 connected to each other (in the sense that there are $u\in\mathfrak{D}_I$, $v\in\mathfrak{D}_J$ such that
 $u\lessdot v$ or $v\gtrdot u$) if and only if one of the following conditions is true:
 \begin{itemize}
  \item $I\subseteq J$ and $\# I = \# J -1$, or
  \item $J\subseteq I$ and $\# J = \# I -1$.
 \end{itemize}
\end{lemm}

\begin{proof}
We compare the descents sets of $\sigma$ and $s_i \sigma$. If $i$ and $i+1$ are not adjacent in $\sigma$, $\Des(\sigma) = \Des(s_i\sigma)$.
Otherwise, either $\Des(\sigma) = \Des(s_i\sigma) \cup \{i\} $ or  $\Des(s_i\sigma) = \Des(\sigma) \cup \{i\} $.
It follows that if $\mathfrak{D}_I$ and $\mathfrak{D}_J$ are connected, $I$ and $J$ satisfy the given conditions.

As for the other implication, suppose $I$ and $J$ satisfy the given conditions, for example $I = J \cup\{i\}$.
It is easy to find $\sigma \in \mathcal{D}_I$ such that $\sigma_i = \sigma_{i+1}+1$.
Then the two permutations $\sigma$ and $s_i\sigma \in \mathcal{D}_J$ show that $\mathfrak{D}_I$ and $\mathfrak{D}_J$ are connected.
\end{proof}


\begin{lemm} \label{connected}
 Each set $\mathfrak{W}\pp_j$ or $\mathfrak{W}\mm_j$ is connected in the left weak order. 
\end{lemm}

\begin{proof}
By symmetry, we only need to consider the case of $\mathfrak{W}\pp_j$.
This set is a union of descent classes. Since each descent class is connected, it remains to show that
these descent classes are connected to each other, using the previous lemma. 

The descent classes included in $\mathfrak{W}\pp_j$ are those indexed by subsets $X\subseteq\{1,\dots,n-1\}$ 
such that:
\begin{itemize}
 \item $X$ is a union of $\lfloor j/2 \rfloor$ intervals and no less (i.e.~the number of descending runs),
 \item $1\notin X$  (since the elements in $\mathfrak{W}\pp_j$ begin with an ascending run),
 \item $n-1\in X$ if and only if $j$ is even (since the first run is ascending, the last run is descending if and only if $j$ is even).
\end{itemize}

By adding or removing elements in these subsets as in the previous lemma,
(and staying in the same class of subsets)
we can reach a particular chosen subset, for example
$\{2,4,6,\dots, 2 \lfloor j/2 \rfloor \}$ if $j$ is odd and 
$\{2,4,6,\dots, 2 \lfloor j/2 \rfloor-1, n-1 \}$ if $j$ is even.
This fact is rather straightforward so we omit the full formal proof.
Connectedness easily follows.
\end{proof}

\subsection{The scheme of proof}


To prove that $\mathcal{W}^\pm$ is an algebra, we show $W\pp_jW\pp_k\in\mathcal{W}^\pm$. 
The other products, i.e.~$W_j\pp W_k\mm $, $W_j\mm W_k\pp$, and $W_j\mm W_k\mm$, can be 
done similarly (or more formally, one can use symmetry properties described in 
Subsection~\ref{subsec:othersubalg}).

For each permutation $\sigma\in\mathfrak{S}_n$, let
\[
  \mathfrak{C}_{j,k}^\sigma = \{ (\alpha,\beta) \in \mathfrak{W}_j\pp \times \mathfrak{W}_k\pp \, : \, \alpha\beta=\sigma \}.
\]
Then $c_{j,k}^\sigma = \# \mathfrak{C}_{j,k}^\sigma $ is the coefficient of $\sigma$ in the product 
$W\pp_j W\pp_k$. Our goal is to prove that $ c_{j,k}^\sigma = c_{j,k}^\tau $ if $\sigma,\tau$ are both in 
$\mathfrak{W}_m\pp$ or both in $\mathfrak{W}_m\mm$. 

We only consider the case where 
$\sigma$ and $\tau$ are neighbours in the left weak order, i.e., either $\sigma \lessdot \tau$
or $\tau \lessdot \sigma$.
Indeed, using the connectedness shown in Lemma~\ref{connected}, this implies the general case. 
So we assume $ s_i \sigma = \tau$ for some $i$.

Let us first consider the case where $i$ and $i+1$ are not adjacent in $\sigma =  \sigma_1 \dots \sigma_n$ 
(and hence in $\tau=s_i\sigma$). 
It implies that $\sigma$ and $\tau$ are in the same descent class. 
But since $W_j^+W_k^+$ is in the descent algebra (as $W_j^+$ and $W_k^+$ both are), the coefficients
of $\sigma$ and $\tau$ are equal, i.e.~$c_{j,k}^\sigma = c_{j,k}^\tau $ as we wanted.

It remains to consider the case where $i$ and $i+1$ are adjacent in $\sigma$ 
(and hence in $\tau$).
We will give a bijection $ \Phi : \mathfrak{C}_{j,k}^\sigma \to \mathfrak{C}_{j,k}^\tau $.
This bijection is rather long to define as we need to consider many cases, and it is
described in the next subsection. Note that
$\sigma$ and $\tau$ have symmetric roles here, since we did not specify which of $\sigma \lessdot \tau$
or $\tau \lessdot \sigma$ is true. So by exchanging the role of $\sigma$ and $\tau$, the definition
in the next subsection also describes another map $ \Psi : \mathfrak{C}_{j,k}^\tau \to \mathfrak{C}_{j,k}^\sigma $.
At each step of the definition, we can check that $\Psi$ is in the inverse map of $\Phi$, so that these maps
are indeed bijections.

\subsection{\texorpdfstring{The bijection $\Phi$}{The bijection Φ}}

\label{subsec:bijphi}

We are now in the case where $i$ and $i+1$ are adjacent in $\sigma$.
So there is an integer $h$ such that $\tau = s_i\sigma=\sigma s_h$,
and also $\sigma( \{h,h+1\} ) = \tau( \{h,h+1\} ) = \{i,i+1\} $. 

Note that for any $u\in\mathfrak{S}_n$, $u$ and $us_1$ cannot be both
in $\mathfrak{W}_m^+$ or both in $\mathfrak{W}_m^-$. The same is true for $u$ and $u s_{n-1}$.
So we have $2\leq h \leq n-2$.

\subsubsection{First case}

In the first case, we assume that $s_i \alpha \in \mathfrak{W}_j\pp$. 
Thus we can define $\Phi((\alpha,\beta)) = (s_i\alpha,\beta)$, this pair being clearly in 
$\mathfrak{C}_{j,k}^\tau$.

Note that the pair $(s_i\alpha,\beta) \in\mathfrak{C}_{j,k}^\tau$ also leads to the first case
when we apply the map $\Psi$, so that $\Psi \circ \Phi$ is the identity when applied on 
$(\alpha,\beta)$ falling in the first case.

\smallskip

From now on, assume $s_i \alpha \notin \mathfrak{W}\pp_j $. It follows that $s_i\alpha$ is not 
in the same descent class as $\alpha$, so $i$ and $i+1$ are adjacent in $\alpha$.
So there is an integer $g$ such that $s_i\alpha = \alpha s_g$. 
We also have $s_g \beta = \beta s_h$ since $\alpha s_g \beta = s_i \alpha \beta = \alpha \beta s_h $.
Also, $\alpha(\{g,g+1\}) = \{i,i+1\}$ and $\beta(\{h,h+1\})=\{ g,g+1 \}$.

\subsubsection{Second case}

In the second case, we assume that $s_g\beta \in \mathfrak{W}_k\pp$. We define 
$\Phi((\alpha,\beta))= (\alpha, s_g\beta) $, which is in $\mathfrak{C}_{j,k}^\tau$ 
(since $\alpha s_g \beta = s_i \alpha\beta = s_i\sigma = \tau$).

Note that the pair $(\alpha, s_g\beta) \in\mathfrak{C}_{j,k}^\tau$ also leads to this second case
when we apply the map $\Psi$, so that $\Psi \circ \Phi$ is the identity when applied on 
$(\alpha,\beta)$ falling in the second case.


\smallskip

So from now on, assume that $s_g\beta \notin \mathfrak{W}_k\pp$. 
This is our last case.

\subsubsection{Third case}

We will extensively use the following fact. The proof is clear upon inspection of a few cases.

\begin{lemm}
 Let $u\in\mathfrak{S}_n$ and $a,b$ such that $s_a u = u s_b$ (so that $u(\{b,b+1\})=\{a,a+1\}$).
 Then $u$ and $s_a u$ are both in $\mathfrak{W}^+_m$ or both in $\mathfrak{W}^-_m$ for some $m$, if and only
 if:
 \begin{itemize}
  \item $2\leq b \leq n-2$, and
  \item $u(b-1)$ and $u(b+2)$ are both $<a$ or both $>a+1$.
 \end{itemize}
\end{lemm}

To begin, let us apply this lemma to $\beta \in \mathfrak{W}^+_k$ and 
$s_g \beta = \beta s_h \notin \mathfrak{W}^+_k$. Since we already know $2\leq h \leq n-2$,
we obtain:
\begin{itemize}
 \item either $\beta(h-1)<g$   and $\beta(h+2)>g+1$, \hfill (A)
 \item or     $\beta(h-1)>g+1$ and $\beta(h+2)<g$.   \hfill (B)
\end{itemize}
Note that this implies $2\leq g \leq n-2$.

Next, we apply the lemma to $\alpha\in\mathfrak{W}^+_j$ and 
$s_i\alpha = \alpha s_g \notin\mathfrak{W}\pp_j$. Since we just got $2\leq g \leq n-2$, it follows:
\begin{itemize}
 \item either $\alpha(g-1)<i$ and $\alpha(g+2)>i+1$,   \hfill (C)
 \item or $\alpha(g-1)>i+1$ and $\alpha(g+2)<i$.       \hfill (D)
\end{itemize}

Next, we apply the lemma to $\sigma$ and $s_i\sigma = \sigma s_h$,
which are both in $\mathfrak{W}\pp_m$. Since $2\leq h \leq n-1$, we get 
(writing $\alpha\beta$ instead of $\sigma$):
\begin{itemize}
 \item either $\alpha(\beta(h-1))<i$ and $\alpha(\beta(h+2))<i$,  \hfill (E)
 \item or $\alpha(\beta(h-1))>i+1$ and $\alpha(\beta(h+2))>i+1$.  \hfill (F)
\end{itemize}

To define $\Phi$, we only need to distinguish between cases C or D, and E or F.
Suppose first that we are in case C-E.

\begin{lemm}
In case C-E as above, there is a unique (word) factorization $\alpha_1\dots\alpha_n = X_1X_2X_3X_4$ where the four factors are 
nonempty and such that: $X_2 = i,i+1$ or $X_2= i+1,i$, $X_3$ contains only letters $>i+1$, the first 
letter of $X_4$ is $<i$. 
\end{lemm}

\begin{proof}
To begin, $X_2$ is defined as the factor containing $i$ and $i+1$ (we have $\alpha(\{g,g+1\})=\{i,i+1\}$).
Since $g\geq2$, we deduce that $X_1$ is nonempty.
Since $\alpha(g+2)>i+1$, i.e.~the first letter to the right of $X_2$ 
is $>i+1$, we can define $X_3$ as the largest factor following $X_2$ and containing letters $>i+1$.

Then, we have $\alpha(\beta(h-1))<i$ and $\alpha(\beta(h+2))<i$. Since at least one of $\beta(h-1)$
and $\beta(h+2)$ is $>g+1$, it follows that there is at least one letter $<i$ to the right of $X_2$.
So $X_4$ is nonempty. 
\end{proof}

Then we define $\Phi((\alpha,\beta)) = (\alpha',\beta')$ as follows: 
\begin{itemize}
 \item $\alpha'= X_1 X_3 \overline{X_2} X_4$ where $\overline{X_2}$ is $X_2$ in reversed order. 
 \item $\beta'$ is such that $\beta'(h) = \beta(h) + |X_3|$, $\beta'(h+1) = \beta(h+1) + |X_3|$,
       if $\alpha(\beta(u))$ appears in $X_3$ then $\beta'(u) = \beta(u)-2$, 
       otherwise (i.e.~$\alpha(\beta(u))$ appears in $X_1$ or $X_4$) then $\beta'(u)=\beta(u)$.
       Here $|X_3|$ is the length of $X_3$.
\end{itemize}

\begin{lemm}
 We have $(\alpha',\beta')\in\mathfrak{C}_{j,k}^{\tau}$.
\end{lemm}

\begin{proof}
Since the last letter of $X_1$ (i.e.~$\alpha(g-1)$) and the first of $X_4$ are both $<i$, and $X_3$ contains 
only letters $>i+1$, we easily get that $\alpha'\in\mathcal{W}^+_j$ (knowing $\alpha\in\mathcal{W}^+_j$).

Among the two letters $\alpha(\beta(h-1))$ and $\alpha(\beta(h+2))$, one appears in $X_1$ and the other in $X_4$.
One can deduce that four letters $\beta'(h-1),\dots,\beta'(h+2)$ are in the same relative order as $\beta(h-1),\dots,\beta(h+2)$.
Then, suppose we are in the case $\beta'(u) = \beta(u)-2$, i.e.~$\alpha(\beta(u))$ appears in $X_3$.
Let $v=u\pm 1$, then $\alpha(\beta(v))$ can appear in $X_1$, $X_3$, or $X_4$. In each case, we check that
$\beta'(u)$ and $\beta'(v)$ are in the same relative order as $\beta(u)$ and $\beta(v)$. So $\beta'$ is in the 
same descent class as $\beta$ and in paticular $\beta'\in\mathfrak{W}^+_k$.

By distinguishing which of the factors $X_1,\dots,X_4$ contains $\alpha(\beta(u))$, from the definition of $\beta'$
we can check that $\alpha'(\beta'(u))= s_i ( \alpha(\beta(u)) )$. So $\alpha' \beta' = s_i \alpha\beta$.
\end{proof}

Let $g'=g+|X_3|$. We have $s_i \alpha' = \alpha' s_{g'}$ and $s_{g'} \beta' = \beta' s_h$ by definition of
$\alpha'$ and $\beta'$. So $(\alpha',\beta')$ falls in the third case when we apply $\Psi$. Moreover:
\begin{itemize}
 \item $\alpha'(g'-1)>i+1$ (since this number is the last letter of $X_3$) 
       , so $(\alpha',\beta')$ falls in case D.
 \item $\alpha'(\beta'(h-1)) = \alpha(\beta(h-1)) <i$ 
       (since $\alpha'\beta' = s_i \alpha\beta$ )
       , so $(\alpha',\beta')$ falls in case E.
\end{itemize}

So $(\alpha',\beta')$ falls in case D-E. Consequently, there is only one way to define $\Phi$
in case D-E in order to get that $\Psi\circ \Phi$ is the identity. Namely, if $(\alpha,\beta)$ falls
in case D-E, we factor $\alpha = X_1 X_2 X_3 X_4 $ where $X_3$ is $i,i+1$ or $i+1,i$, $X_2$
contains only letters $>i+1$, the last letter of $X_1$ is $<i$. Once again the factorization exists
and is unique. Then we define $\Phi((\alpha,\beta)) = (\alpha',\beta')$ as follows: 
\begin{itemize}
 \item $\alpha'= X_1 \overline{X_3} X_2 X_4$. 
 \item $\beta'$ is such that $\beta'(h) = \beta(h) - |X_2|$, $\beta'(h+1) = \beta(h+1) - |X_2|$,
       if $\alpha(\beta(u))$ appears in $X_2$ then $\beta'(u) = \beta(u)+2$, 
       otherwise (i.e.~$\alpha(\beta(u))$ appears in $X_1$ or $X_4$) then $\beta'(u)=\beta(u)$.
\end{itemize}
From the definitions, we check that $\Psi\circ\Phi$ is the identity on all $(\alpha,\beta)$ that
falls in case C-E or D-E.

It remains to define $\Phi$ and $\Psi$ in a similar manner to exchange the case C-F with the case D-F.
This is completely similar except that we exchange the conditions ``$<i$'' and ``$>i+1$''.

This ends the definition of the bijection, hence of the proof of $W_j^+ W_k^+ \in \mathcal{W}^\pm$.

\subsection{The case of the other subalgebras}

\label{subsec:othersubalg}

Knowing that $\mathcal{W}^\pm$ is an algebra, we can finish the proof of Theorem~\ref{subalg}.

\begin{prop}
 $\mathcal{W}^\circ$ is a subalgebra of $\mathcal{W}^\pm$.
\end{prop}

\begin{proof}
For each permutation $\sigma$, we define:
\[
  \mathfrak{D}^{\sigma}_{j,k} = \{ (\alpha,\beta) \in \mathfrak{W}^\circ_j \times \mathfrak{W}^\circ_k \,: \, \alpha\beta = \sigma \},
\]
and $d^{\sigma}_{j,k} = \# \mathfrak{D}^{\sigma}_{j,k}$. If $\sigma\in\mathfrak{W}_m^+$ or
$\sigma\in\mathfrak{W}_m^-$, then $d^{\sigma}_{j,k}$ is the coefficient of $W_m^+$ or $W_m^-$, respectively, in 
$W_j^\circ W_k^\circ $.
We will show that for each $m$ such that $2\leq m \leq n-1$, we have $d^{\sigma}_{j,k} = d^{\tau}_{j,k} $ 
for some well-chosen $\sigma\in\mathfrak{W}^+_{m}$ and $\tau\in\mathfrak{W}^-_{m-1}$. Indeed
this implies $ W_j^\circ W_k^\circ \in \mathcal{W}^\circ$.

It is easy to find an element $\sigma \in \mathfrak{W}_m^+$ such that $\sigma_1 = n-1$ and $\sigma_2 = n$.
We deduce that $\tau = s_{n-1}\sigma = \sigma s_1$ is in $\mathfrak{W}_{m-1}^-$.
Then, it remains only to find a bijection beween $\mathfrak{D}^{\sigma}_{j,k}$ and $\mathfrak{D}^{\tau}_{j,k}$.

This bijection $\Delta$ is defined as follows. 
Let $(\alpha,\beta) \in \mathfrak{D}^{\sigma}_{j,k}$.
In the first case, if $s_{n-1}\alpha \in \mathfrak{W}^\circ_j$, 
we set $\Delta(\alpha,\beta) = (s_{n-1}\alpha,\beta) $. 

Otherwise, it easily follows $s_{n-1}\alpha = \alpha s_{n-1}$. In the second case, if
$s_{n-1}\beta \in \mathfrak{W}^\circ_k $,
we set $\Delta(\alpha,\beta) = (\alpha,s_{n-1}\beta)$.

There is in fact no other case to consider: if we are not in the first case, we have 
necessarily $s_{n-1}\beta \in \mathfrak{W}^\circ_k $. Otherwise, 
we would have $s_{n-1}\beta = \beta s_{n-1}$, and it would follow that $s_{n-1}\sigma = \sigma s_{n-1}$ (this is a contradiction to 
$s_{n-1}\sigma=\sigma s_1$).
\end{proof}

As for the other subalgebras, we don't need to do a similar bijection, as it is more
convenient to use some symmetry properties. Let $\omega := W_1\mm = n\dots 21$.
We clearly have $\run(\omega \sigma ) = \run( \sigma \omega ) = \run( \sigma ) $ for
any $\sigma \in\mathfrak{S}_n$. 
Moreover, $1\in\Des(\sigma)$ iff $1\notin \Des(\omega\sigma)$, so that:
\[
   \omega W_k\pp = W_k\mm, \qquad \omega W_k\mm = W_k\pp.
\]

Note that if $\sigma\in\mathfrak{S}_n$ has an odd 
number of alternating runs, the first and last runs are of similar nature
(both ascending or both descending), but if $\sigma$ has an even number of runs, 
they are of different nature (one is ascending, the other descending).
It is then easy to read the alternating runs of $\sigma\omega = \sigma_n \dots \sigma_1$
from those of $\sigma = \sigma_1 \dots \sigma_n$. It follows:
\begin{equation}   \label{relwW}
   W_k\pp \omega = \begin{cases} W_k\mm \text{ if } k \text{ is odd}, \\
                                 W_k\pp \text{ if } k \text{ is even},
                   \end{cases}
 \qquad 
   W_k\mm \omega = \begin{cases} W_k\pp \text{ if } k \text{ is odd}, \\
                                 W_k\mm \text{ if } k \text{ is even},
                   \end{cases}
\end{equation}

Note that this implies
\begin{align*} 
  W_j\mm W_k\pp &= \omega W_j\pp W_k\pp\\
  W_j\pp W_k\mm &= \begin{cases}
                      W_j\pp W_k\pp & \text{ if } j \text{ is even,} \\
                      \omega W_j\pp W_k\pp & \text{ if } j \text{ is odd,}
                   \end{cases} \\
  W_j\mm W_k\mm &= \begin{cases}
                      W_j\pp W_k\pp & \text{ if } j \text{ is odd,} \\
                      \omega W_j\pp W_k\pp & \text{ if } j \text{ is even.}
                   \end{cases}
\end{align*}
As a consequence, to show that $\mathcal{W}^\pm$ is an algebra
it suffices to show that $W\pp_j W\pp_k \in \mathcal{W}^\pm $, as was mentioned earlier.

%
%

\begin{prop}
 $\mathcal{W}$ is a subalgebra of $\mathcal{W}^\pm$.
\end{prop}

\begin{proof}
From the elementary properties of $\omega = W_1\mm$ (and $W_1^+$, the unit element), 
we have $W_1 W_k\pp = W_1 W_k\mm = W_k $.
It follows that $\mathcal{W}$ is the right ideal $W_1 \mathcal{W}^\pm$ and {\it a fortiori}
a subalgebra.
\end{proof}

\begin{rema}
 The previous proposition should be understood in the sense of nonunital algebras, since
 $\mathcal{W}$ does not contain contain the unit of $\mathcal{W}^\pm$. 
 Still, $\mathcal{W}$ has a unit element as an abstract algebra, namely $\frac{1}{2} W_1$
 (at the condition of allowing rational coefficients).
\end{rema}

\begin{prop}
 $\mathcal{C}$ is a subalgebra of $\mathcal{W}^\pm$.
\end{prop}

\begin{proof}
From the elementary properties of $\omega = W_1\mm$, we have:
\begin{align*}
  \omega W_k\pp \omega = \begin{cases} W_k\pp \text{ if } k \text{ is odd}, \\
                                       W_k\mm \text{ if } k \text{ is even}, 
                         \end{cases}
  \qquad 
  \omega W_k\mm \omega = \begin{cases} W_k\mm \text{ if } k \text{ is odd},  \\
                                       W_k\pp \text{ if } k \text{ is even}.
                         \end{cases}
\end{align*}
Using that $x \mapsto \omega x \omega$ is a linear symmetry, it follows that
\begin{align*}
   \{ x\in\mathcal{W}^\pm \,:\,  \omega x \omega =  x  \} 
       = \Vect (W^+_{2k+1}, W^-_{2k+1}, W_{2k}) = \mathcal{C}.
\end{align*}
If $x,y \in \mathcal{C}$, we have $\omega xy \omega = \omega x \omega ^2 y \omega =xy $,
so $\mathcal{C}$ is closed under product.
\end{proof}

We turn to the case of peak algebras.

\begin{lemm} We have:
\begin{align}  \label{relPW}
   P_k       &= W^+_{2k} + W^-_{2k+1} + W^+_{2k+1} + W^-_{2k+2}, \\    \label{relPW2}
   P_k^\circ &= W^-_{2k-1} + W^+_{2k} + W^-_{2k} + W^+_{2k+1}.
\end{align}
\end{lemm}
\begin{proof}
For the first equality, consider a permutation $\sigma$ with $\pk(\sigma)=k$. It has at least $2k$ runs,
since each peak follows an ascending run, and precedes a descending run. 
Besides these $2k$ runs, whose first is ascending, there is possibly an additional descending run
at the beginning, and possibly an additional ascending run at the end. 
The four possibilities give the four terms in the sum.

As for the second equality, consider a permutation  $\pk^\circ(\sigma)=k$. It has at least $2k-1$ runs,
since now the first peak might not follow an ascending run (if it is equal to 1).
Besides these $2k-1$ runs, whose first is descending, there is possibly an additional ascending run
at the beginning, and possibly an additional ascending run at the end.
Once again, the four possibilities give the four terms in the sum.
\end{proof}

\begin{prop}
$\mathcal{P}$ and $\mathcal{P}^\circ$ are subalgebras of $\mathcal{W}^\pm$.
\end{prop}

\begin{proof}
The previous lemma shows that $\mathcal{P}$ and $\mathcal{P}^\circ$ are subspaces 
of $\mathcal{W}^\pm$, so it remains to show that they are closed under product.

We begin with the case of $\mathcal{P}^\circ$. From the previous lemma, we have
$P^\circ_k =  W^\circ_{2k} + W^\circ_{2k+1}$. Now let $x\in\mathcal{W}^\circ$, then
it is in $\mathcal{P}^\circ$ iff the coefficients of $W_{2k}^\circ$ and $W_{2k+1}^\circ$
are the same for all $k$. This is also equivalent to the equality of the coefficients
of $W_{2k}\pp$ and $W_{2k}\mm$ in $x$, i.e.~to the condition $x \in \mathcal{C}$.
So we have $\mathcal{P}^\circ = \mathcal{W}^\circ \cap \mathcal{C}$, and this is a subalgebra
as an intersection of subalgebras.

In the case of $\mathcal{P}$, we first note that it is closed under product iff  
$\omega \mathcal{P} \omega$ is also closed under product. We have:
\[
   \omega P_k \omega = W^-_{2k} + W^-_{2k+1} + W^+_{2k+1} + W^+_{2k+2} 
      = W^\circ_{2k+1} + W^\circ_{2k+2},
\]
so $\omega \mathcal{P} \omega \subseteq \mathcal{W}^\circ $.
Now let $x\in\mathcal{W}^\circ$, then it is in $\omega \mathcal{P} \omega$ iff the 
coefficients of $W_{2k+1}^\circ$ and $W_{2k+2}^\circ$ are the same for all $k$. 
This is also equivalent to the equality of the coefficients
of $W_{2k+1}\pp$ and $W_{2k+1}\mm$ in $x$, and also to the condition $x\omega=x$.
So we have $ \omega \mathcal{P} \omega = \{ x \in \mathcal{W}^\circ \, : \, x\omega = x\}$, 
which is clearly closed under product.
\end{proof}

\section{Algebraic proof for the existence of run algebras}
\label{sec:algproof}

This section gives a second proof, using noncommutative symmetric
functions, that grouping permutations by the number of runs gives
an algebra. This will require working in all degrees $n$ at the same
time, so from now on add a second index to all previous quantities
indicating the degree, if these quantities involve one degree only.
For example, we now write $\mathcal{W}_{n}$ for the commutative run
algebra in degree $n$, and $W_{k,n}$ for its basis.

First define the following elements in $\widehat{Sym}$: 
\[
  V_{1}\pp:=\sum_{i\geq0}R_{i}=\sum_{n\geq 0}W_{1,n}\pp,\qquad V_{1}\mm:=\sum_{i\geq0}R_{(1^{i})}=\sum_{n\geq 0}W_{1,n}\mm.
\]
The intuition is that $V_{1}\pp$ is an ascending run and $V_{1}\mm$
is a descending run. 
To obtain multiple runs, one concatenates ascending runs and descending
runs alternately. Since external product roughly expresses concatenation,
it is no surprise that we define 
\begin{align*}
V_{2k}\pp & :=(V_{1}\pp\star V_{1}\mm)^{\star k}, & V_{2k}\mm & :=(V_{1}\mm\star V_{1}\pp)^{\star k},\\
V_{2k+1}\pp & :=V_{1}\pp\star(V_{1}\mm\star V_{1}\pp)^{\star k},\qquad & V_{2k+1}\mm & :=V_{1}\mm\star(V_{1}\pp\star V_{1}\mm)^{\star k}, 
\end{align*}
and
\begin{align*}
V_{k} & :=V_{k}\pp+V_{k}\mm.
\end{align*}
Write $V_{k,n}\pp$ for the degree $n$ part of $V_{k}\pp$, and similarly
for $V_{k,n}\mm,V_{k,n}$, so 
\[
V_{k}\pp=\sum_{n\geq0}V_{k,n}\pp,\qquad V_{k}\mm=\sum_{n\geq0}V_{k,n}\mm,\qquad V_{k}=\sum_{n\geq0}V_{k,n}.
\]
Note that $V_{k,n}\pp,V_{k,n}\mm$ are non-zero even when $k>n$. 

Recall from Section~\ref{nsym} that $V_{1}\pp,V_{1}\mm$ are grouplike. Being
a product of grouplike elements, $V_{k}\pp,V_{k}\mm$ are also grouplike.
Thus the simplified splitting formula in Section~\ref{nsym} holds, and this
allows a straightforward proof that the $V$s are ``multiplicative''
(under the internal product) in the following way:
\begin{prop} \label{propmulV}
We have:
\[
V_{k}\pp V_{l}\pp=V_{kl}\pp,\qquad V_{k}\mm V_{l}\pp = V_{kl}\mm.
\]
If $k$ is even, then
\[
  V_{k}\pp V_{l}\mm = V_{kl}\pp, \qquad V_{k}\mm V_{l}\mm =V_{kl}\mm.
\]
If $k$ is odd, then
\[
  V_{k}\pp V_{l}\mm = V_{kl}\mm, \qquad V_{k}\mm V_{l}\mm = V_{kl}\pp.
\]
\end{prop}
\begin{proof}
The splitting formula asserts that, for any $F_1,\dots,F_l$ in $\widehat{\Sym}$, it holds that 
\[
  V_{k}\pp (F_{1}\star\dots\star F_{l}) = (V_{k}\pp F_{1})\star\dots\star(V_{k}\pp F_{l}),
\]
and similarly for $V_{k}\mm$. 
So in particular:
\begin{align*}
V_{k}\pp V_{l}\pp =(V_{k}\pp V_{1}\pp)\star(V_{k}\pp V_{1}\mm)\star\dots\star(V_{k}\pp V_{1}^\pm)
\end{align*}
where the $\pm$ is $+$ or $-$ depending on the parity of $l$.
For even $k$, once we know that $ V_{k}\pp V_{1}\pp = V_{k}\pp V_{1}\mm =V_{k}\pp$ (the case $l=1$), this simplifies to 
\begin{align*}
V_{k}\pp V_{l}\pp  = V_{k}\pp\star V_{k}\pp\star\dots\star V_{k}\pp = V_{kl}^+.
\end{align*}
For odd $k$, once we know that $ V_{k}\pp V_{1}\pp = V_k^+$ and $ V_{k}\pp V_{1}\mm =V_{k}\mm$ (the case $l=1$), this simplifies to 
\begin{align*}
V_{k}\pp V_{l}\pp  = V_{k}\pp\star V_{k}\mm \star \dots \star V_{k}^\pm = V_{kl}^+.
\end{align*}
The other products are treated similarly, and we are left with the case $l=1$.

Recall $V_{1}\pp=\sum_{i}R_{i}$ and $R_{i}$ is the identity for
the internal product, so $V_{k}\pp V_{1}\pp=V_{k}\pp$, and $V_{k}\mm V_{1}\pp=V_{k}\mm$.

Now $V_{1}\mm=\sum_{i\geq0}R_{1^{i}}$ and $R_{1^i}$ is the maximal permutation $\omega_i = i,i-1,\dots,3,2,1$. 
Internal product to the right by $\omega_i$
sends a permutation $\sigma$ to its reversal $\sigma_i,\dots,\sigma_1$, which exchanges the ascents and descents and reverses their order. Given a composition $I$, let $\bar{I}$ denote the result after this exchanging and reversing, e.g. $(1,3,2)\mapsto (2,1,2,1) \mapsto (1,2,1,2)$. Then $R_I V_1\mm=R_{\bar{I}}$, and 
\begin{align*}
(R_I \star R_J)V_1\mm &= (R_{(i_{1},\dots,i_{\ell(I)},j_{1},\dots,j_{\ell(J)})}+R_{(i_{1},\dots,i_{\ell(I)-1},i_{\ell(I)}+j_{1},j_{2},\dots,j_{\ell(J)})})V_1\mm \\
&= R_{\overline{(i_{1},\dots,i_{\ell(I)}+j_{1},\dots,j_{\ell(J)})}}+R_{\overline{(i_{1},\dots,i_{\ell(I)-1},i_{\ell(I)},j_{1},j_{2},\dots,j_{\ell(J)})}}\\
&=R_{\bar{J}}\star R_{\bar{I}} = (R_{J} V_{1}\mm) \star ( R_{I}V_{1}\mm ).
\end{align*}
By linearity, it follows that, for any $G_1,G_2\in\widehat{\Sym}$, we have $(G_{1}\star G_{2})V_{1}\mm =  (G_{2} V_{1}\mm) \star ( G_{1}V_{1}\mm )$, and similarly
$(G_{1}\star\dots\star G_{l})V_{1}\mm =  (G_{l} V_{1}\mm) \star\dots\star ( G_{1}V_{1}\mm )$
Also we have $V_{1}\mm V_{1}\mm=V_{1}\pp$, since reversing a permutation twice does not change it. As a result,
if $k$ is even, then
\begin{align*}
   V_{k}\pp V_{1}\mm &=( V_{1}\mm V_1\mm)\star ( V_{1}\pp V_1\mm) \star \dots \star ( V_{1}\pp V_1\mm) = V_{1}\pp \star V_1\mm \star \dots \star  V_1\mm = V_k\pp; \\
V_{k}\mm V_{1}\mm &=( V_{1}\pp V_1\mm)\star ( V_{1}\mm V_1\mm) \star \dots \star ( V_{1}\mm V_1\mm) = V_{1}\mm \star V_1\pp \star \dots \star  V_1\pp = V_k\mm.
\end{align*}

And if $k$ is odd, then
\begin{align*}
   V_{k}\pp V_{1}\mm &=( V_{1}\pp V_1\mm)\star ( V_{1}\mm V_1\mm) \star \dots \star ( V_{1}\pp V_1\mm) = V_{1}\mm \star V_1\pp \star \dots \star  V_1\mm = V_k\mm; \\
V_{k}\mm V_{1}\mm &=( V_{1}\mm V_1\mm)\star ( V_{1}\pp V_1\mm) \star \dots \star ( V_{1}\mm V_1\mm) = V_{1}\pp \star V_1\pp \star \dots \star  V_1\pp = V_k\pp.
\end{align*}
(Note that the sign rule here is the same as in Equation~\eqref{relwW}.
This will be explained by the relation between Vs and Ws given below.)
\end{proof}

Now it is clear that $\Vect(V_{k,n}\pp,V_{k,n}\mm:k\in\mathbb{N})$
is closed under the internal product, and it is easy to identify subspaces,
such as $\Vect(V_{k,n}\pp+V_{k,n}\mm:k\in\mathbb{N})$, which
form subalgebras. It remains to equate these to the run algebras $\mathcal{W}_{n}^{\pm}$
and $\mathcal{W}_{n}$, and see how the Vs give them new bases.
We will do this by relating
$V_{2k}\mm$ and $V_{2k+1}\pp$ to the Eulerian peak and left peak
algebras respectively, via a triangular change of basis.
\begin{prop}
For all $k$, we have $V_{2k,n}\mm\in\mathcal{P}_{n}$, the Eulerian
peak algebra. Moreover, if $k\leq\frac{n+1}{2}$, then 
\[
V_{2k,n}\mm=2^{2k-1}P_{k-1,n}+\sum_{l<k-1}a_{k,l,n}P_{l,n}
\]
 for some constants $a_{k,l,n}$. Hence $\{V_{2k,n}\mm: 1\leq  k \leq \lfloor \frac{n+1}{2} \rfloor \}$
is a basis for $\mathcal{P}_{n}$.\end{prop}
\begin{proof}
For the first assertion, it suffices to show that, for each composition
$I$, the coefficient of $R_{I}$ in $V_{2k,n}\mm$ depends only on
the number of peaks in any permutation with descent composition $I$ - let $\sigma_I$ denote such a permutation.

First observe that $V_{2k}\mm=(V_{2}\mm)^{\star k}$ by definition,
and that $V_{2}\mm=1+2\sum R_{I}$ over all compositions $I$ where $\sigma_I$ has
no peaks.

Recall that 
\[
R_{(i_{1},\dots,i_{\ell(I)})}\star R_{(j_{1},\dots,j_{\ell(J)})}=R_{(i_{1},\dots,i_{\ell(I)},j_{1},\dots,j_{\ell(J)})}+R_{(i_{1},\dots,i_{\ell(I)-1},i_{\ell(I)}+j_{1},j_{2},\dots,j_{\ell(J)}).}
\]
So the coefficient of $R_{I}$ in any external product $F\star G$
of noncommutative symmetric functions is the sum, over all $J:=(j_{1},\dots,j_{\ell(J)})$,
$J':=(j'_{1}, \dots, j'_{\ell(J')})$ with $I=(j_{1}, \dots, j_{\ell(J)}, j'_{1}, \dots, j'_{\ell(J')})$
or $I=(j_{1},\dots,j_{\ell(J)-1},j_{\ell(J)}+j'_{1},j'_{2}, \dots,  \allowbreak  j'_{\ell(J')})$,
of the products of the coefficient of $R_J$ in $F$ with the coefficient
of $R_{J'}$ in $G$. Hence the coefficient of $R_{I}$ in $V_{2k,n}\mm$
is the weighted number of ways to ``cut'' $\sigma_I$ into $k$ pieces,
possibly trivial, so that each piece contains no peak. The weight
is $2^{k'}$, where $k'$ is the number of non-trivial pieces
(because of the coefficient 2 in $V_{2}\mm=1+2\sum R_{I}$).

Since each piece must be peakless, there must be a cut immediately
to the left or immediately to the right of any peak; other than this
the positions of cuts are unconstrained. These pair of positions,
for any set of peaks, are disjoint, so the number of ways to cut $\sigma_I$
into $k'$ non-trivial peakless pieces depends only on the number
of peaks in $\sigma_I$.

If $\sigma_I$ has more than $k-1$ peaks, then $\sigma_I$ cannot be cut into $k$
peakless pieces, so $R_{I}$ does not appear in $V_{2k}\mm$.

If $\sigma_I$ has exactly $k-1$ peaks, then the ways to cut $\sigma_I$ into $k$
peakless pieces are precisely when there is one cut either immediately
to the left or immediately to the right of each peak. There are $2^{k-1}$
such ways, and the weight of this cut is $2^{k}$, since no piece
is trivial. Hence the coefficient of $P_{k-1,n}$ in $V_{2k,n}\mm$
is $2^{2k-1}$.
\end{proof}
Next, we apply a similar argument to $V_{2k+1,n}\pp$.
\begin{prop}
For all $k$, we have $V_{2k+1,n}\pp\in\mathcal{P}_{n}^{\circ}$,
the Eulerian left peak algebra. Moreover, if $k\leq\frac{n}{2}$,
then 
\[
V_{2k+1,n}\pp=2^{2k}P_{k,n}^{\circ}+\sum_{l<k}a_{k,l,n}^{\circ}P_{l,n}^{\circ}
\]
 for some constants $a_{k,l,n}^{\circ}$. Hence $\{V_{2k+1,n}\pp: 0 \leq k \leq \lfloor \frac{n}{2} \rfloor \}$
is a basis for $\mathcal{P}_{n}^{\circ}$.\end{prop}
\begin{proof}
By the same logic as above, the coefficient of $R_{I}$ in $V_{2k+1,n}\pp$
is the weighted number of ways to ``cut'' $\sigma_I$ into $k+1$ pieces,
possibly trivial, so that the leftmost piece has no descents, and
the $k$ other pieces each contain no peak. The weight is $2^{k'}$,
where $k'$ is the number of non-trivial pieces among the $k$ peakless
pieces.

Since each piece must be peakless, there must be a cut immediately to the
left or immediately to the right of all peaks. Additionally, if $\sigma_I$ has a descent in position 1, then the first cut must be immediately before or after position 1, because the first piece must not contain descents. Hence the two conditions may be summarised as requiring a cut immediately to the
left or immediately to the right of all left peaks.

As before, these positions, for any set of left peaks, are disjoint, so the
number of ways to cut into $\sigma_I$ into a descentless first piece and
$k'$ further non-trivial peakless pieces depends only on the number
of left peaks in $\sigma_I$.

If $\sigma_I$ has more than $k$ left peaks, then $\sigma_I$ cannot be cut into
a descentless first piece and $k$ further peakless pieces, so $R_{I}$
does not appear in $V_{2k+1}\pp$.

If $\sigma_I$ has exactly $k$ left peaks, then the ways to cut $\sigma_I$ into
a descentless first piece and $k$ further peakless pieces are precisely
when there is one cut either immediately to the left or immediately
to the right of each left peak. There are $2^{k}$ such ways, and
the weight of this cut is $2^{k}$, since none of the peakless pieces
are trivial. Hence the coefficient of $P_{k,n}^{\circ}$ in $V_{2k+1,n}\pp$
is $2^{2k}$.
\end{proof}

From the above two propositions, and Equations~\eqref{relPW} and~\eqref{relPW2}
relating $P_{k,n}$ and $P_{k,n}^{\circ}$ to $W_{k,n}\pp$, $W_{k,n}\mm$, we
see that:
\begin{align*}
V_{2k,n}\mm & =2^{2k-1}W_{2k,n}\mm+\sum_{l<k}a_{k,l,n}W_{2l+1,n}+a_{k,l,n}W_{2l,n}\pp+a_{k,l-1,n}W_{2l,n}\mm;\\
V_{2k+1,n}\pp & =2^{2k}W_{2k+1,n}\pp+\sum_{l\leq k}a_{k,l,n}^{\circ}W_{2l,n}+a_{k,l-1,n}^{\circ}W_{2l-1,n}\pp+a_{k,l,n}^{\circ}W_{2l-1,n}\mm.
\end{align*}
By exchanging ascending and descending runs, it is also true that
\begin{align*}
V_{2k,n}\pp & =2^{2k-1}W_{2k,n}\pp+\sum_{l<k}a_{k,l,n}W_{2l+1,n}+a_{k,l,n}W_{2l,n}\mm+a_{k,l-1,n}W_{2l,n}\pp;\\
V_{2k+1,n}\mm & =2^{2k}W_{2k+1,n}\mm+\sum_{l\leq k}a_{k,l,n}^{\circ}W_{2l,n}+a_{k,l-1,n}^{\circ}W_{2l-1,n}\mm+a_{k,l,n}^{\circ}W_{2l-1,n}\pp.
\end{align*}
Also, because $P_{k}^{\circ}=W_{2k}^{\circ}+W_{2k+1}^{\circ}$, and
\begin{align*}
P_{k} & =W_{2k}^{+}+W_{2k+1}^{-}+W_{2k+1}^{+}+W_{2k+2}^{-},\\
W_{2k+1}^{\circ}+W_{2k+2}^{\circ} & =W_{2k}^{-}+W_{2k+1}^{+}+W_{2k+1}^{-}+W_{2k+2}^{+},
\end{align*}
we have 
\begin{align*}
V_{2k+1,n}\pp & =2^{2k}W_{2k+1}^{\circ}+2^{2k}W_{2k}^{\circ}+\sum_{l<k}a_{k,l,n}^{\circ}(W_{2l+1,n}^{\circ}+W_{2l,n}^{\circ});\\
V_{2k,n}\pp & =2^{2k-1}W_{2k}^{\circ}+2^{2k-1}W_{2k-1}^{\circ}+\sum_{l<k}a_{k,l-1,n}^{\circ}(W_{2l,n}^{\circ}+W_{2l-1,n}^{\circ}).
\end{align*}
Combining all this with the multiplicativity of the $V_{k}$'s in
Proposition~\ref{propmulV}, we can conclude:

\begin{theo} \label{th:subalgs}
For each $n$, 
\[
\Vect(V_{k,n}\pp,V_{k,n}\mm:k\in\mathbb{N})=\Vect(V_{k,n}\pp,V_{k,n}\mm:1\leq k<n)=\mathcal{W}_{n}^{\pm}
\]
is an algebra under the internal product. Five commutative subalgebras
of $\mathcal{W}_{n}^{\pm}$ are:
\begin{align*}
\Vect(V_{k,n}:k\in\mathbb{N}) & =\Vect(V_{k,n}:1\leq k<n) &  & =\mathcal{W}_{n},\\
\Vect(V_{k,n}\pp:k\in\mathbb{N}) & =\Vect(V_{k,n}\pp:1\leq k\leq n) &  & =\mathcal{W}_{n}^{\circ},\\
\Vect(V_{2k,n}\mm:k\in\mathbb{N}) & =\Vect(V_{2k,n}\mm:1\leq2k\leq n+1) &  & =\mathcal{P}_{n},\\
\Vect(V_{2k+1,n}\pp:k\in\mathbb{N}) & =\Vect(V_{2k+1,n}\pp:1\leq2k+1\leq n+1) &  & =\mathcal{P}_{n}^{\circ},
\end{align*}
and
\begin{align*}
\Vect(V_{2k+1,n}\pp,V_{2k+1,n}\mm,V_{2k,n}:k\in\mathbb{N}) &= \\ \Vect(V_{2k+1,n}\pp,V_{2k+1,n}\mm,V_{2l,n}:1\leq2k+1,2l<n) &= \mathcal{C}_n.
\end{align*}
Furthermore, $\mathcal{P}_{n}^{\circ}$ is central in $\mathcal{W}_{n}^{\pm}$.
\end{theo}

\begin{rema}The quantities $V_{k,n}\pp,V_{k,n}\mm$ previously appeared in~\cite{petersen}, as order polynomials of enriched $P$-partitions. The translation between his notation and ours is
\begin{align*}
\rho(x) & =V_{2x,n}\mm;\\
\bar{\rho}(x) & =V_{2x,n}\pp;\\
\rho^{(l)}(x) & =V_{2x+1,n}\pp;\\
\rho^{(r)}(x) & =V_{2x+1,n}\mm.
\end{align*}
Our Proposition~\ref{propmulV}, the multiplication rule, is his Theorems 3.1, 3.3 and 3.5.
\end{rema}

\section{Computation of the orthogonal idempotents}
\label{sec:idem}

In this section, we compute a complete set of primitive idempotents
for $\mathcal{W}_{n}^{\pm}$ and for its five commutative subalgebras
in Theorem~\ref{th:subalgs}. 


\subsection{Orthogonal idempotents of the commutative subalgebras}

The starting point is the orthogonal idempotents of the Eulerian peak
and left peak algebras. These were computed in \cite{petersen,schocker},
but we rederive them here so this paper is self-contained. The argument
here mirrors the approach of Loday (see \cite[Sec. 5.3]{gelfand}) for the idempotents of
the Eulerian (descent) algebra. The key is the following observation:

 \begin{lemm}\label{multfamilytoidempotents}For some multiplicatively-closed
subset $I\subseteq\mathbb{N}$, suppose $\{X_{i}:i\in I\}$ generate
a subalgebra $\mathcal{A}$ of $\widehat{\Sym}$ with $X_{i}X_{j}=X_{ij}$ (under the
internal product). Assume also that there are elements $J_{k}$ in
this subalgebra such that, for each $i$, we have $X_{i}=\sum_{k\geq0}i^{k}J_{k}$
(as an identity of formal power series within $\widehat{\Sym}$). Then
the non-zero $J_{k}$ form a basis of orthogonal idempotents for $\mathcal{A}$. \end{lemm}

\begin{proof}
For all $i,j\in I$, we have 
\[
\sum_{k}(ij)^{k}J_{k}=X_{ij}=X_{i}X_{j}=\sum_{r,s}i^{r}j^{s}J_{r}J_{s}.
\]
Equating coefficients of $(ij)^{k}$ then shows that $J_{r}J_{s}=0$
if $r\neq s$, and $J_{k}^{2}=J_{k}$.
\end{proof}
Recall from Theorem~\ref{th:subalgs} that the Eulerian peak algebra
is generated by $\{V_{i,n}\mm:i\mbox{ even}\}$ and the Eulerian left
peak algebra by $\{V_{i,n}\pp:i\mbox{ odd}\}$. By Proposition~\ref{propmulV},
these generating sets satisfy the multiplicative condition of Lemma~\ref{multfamilytoidempotents}.
Thus, to calculate the orthogonal idempotents, it suffices to find
$I_{k,n}\mm,J_{k,n}\pp$ such that $V_{2i,n}\mm=\sum_{k\geq0}(2i)^{k}I_{k,n}\mm$
and $V_{2i+1,n}\pp=\sum_{k\geq0}(2i+1)^{k}J_{k,n}\pp$. Such expressions
come from the formal exponential and logarithm operations in $\widehat{\Sym}$,
defined via their familiar power series expansions:
\begin{align*}
\log_{\star}(1+F): & =F-\frac{1}{2}F{}^{\star2}+\frac{1}{3}F{}^{\star3}-\dots,\\
\exp_{\star}F: & =1+F+\frac{1}{2!}F{}^{\star2}+\frac{1}{3!}F{}^{\star3}+\dots,
\end{align*}
for $F\in\widehat{\Sym}$ with no term in degree 0. It can be checked
that $\exp_{\star}(F+G)=(\exp_{\star}F)\star(\exp_{\star}G)$ whenever
$F$ and $G$ commute. Also, $(1+F)^{\star k}=\exp_{\star}(k\log_{\star}(1+F))$
for all positive integers $k$ (and all $F\in\widehat{\Sym}$ with no
term in degree 0), and this can be used to define $(1+F)^{\star k}$
when $k$ is not a positive integer.

Now we have
\begin{align*}
V_{2i}\mm=(V_{2}\mm)^{\star i} & =\exp_{\star}(i\log_{\star}V_{2}\mm)\\
 & =\sum_{k\geq0}\frac{1}{k!}(i\log_{\star}V_{2}\mm)^{\star k}\\
 & =\sum_{k\geq0}(2i)^{k}\left(\frac{1}{2^{k}k!}(\log_{\star}V_{2}\mm)^{\star k}\right),
\end{align*}
and 
\begin{equation} \label{Voddppidems}
\begin{split}
V_{2i+1}\pp & =V_{1}\pp\star(V_{2}\mm)^{\star\left(-\frac{1}{2}\right)}(V_{2}\mm)^{\star\left(i+\frac{1}{2}\right)} \\ & =V_{1}\pp\star(V_{2}\mm)^{\star\left(-\frac{1}{2}\right)}\star\exp_{\star}\left(\left(i+\frac{1}{2}\right)\log_{\star}V_{2}\mm\right)\\
 & =\sum_{k\geq0}(2i+1)^{k}\left(\tfrac{1}{2^{k}k!}V_{1}\pp\star(V_{2}\mm)^{\star\left(-\frac{1}{2}\right)}\star(\log_{\star}V_{2}\mm)^{\star k}\right).
\end{split}
\end{equation}

This motivates the definitions: 
\begin{align*}
I_{1}\mm & :=\frac{1}{2}\log_{\star}V_{2}\mm, & I_{k}\mm & :=\frac{1}{k!}I_{1}\mm{}^{\star k},\\
I_{1}\pp & :=\frac{1}{2}\log_{\star}V_{2}\pp, & I_{k}\pp & :=\frac{1}{k!}I_{1}\pp{}^{\star k},\\
J_{0}\pp & :=V_{1}\pp\star(V_{2}\mm)^{\star\left(-\frac{1}{2}\right)},\qquad & J_{k}\pp & :=J_{0}\pp\star I_{k}\mm,\\
J_{0}\mm & :=V_{1}\mm\star(V_{2}\pp)^{\star\left(-\frac{1}{2}\right)},\qquad & J_{k}\mm & :=J_{0}\mm\star I_{k}\pp.
\end{align*}
Also, make the convention $I_{0}\mm=I_{0}\pp=1$. As with the $V$s,
write $I_{k,n}\pp$, $I_{k,n}\mm$, $J_{k,n}\pp$, and $J_{k,n}\mm$ for the degree
$n$ part of $I_{k}\pp$, $I_{k}\mm$, $J_{k}\pp$,  and $J_{k}\mm$. The above calculations
show that $I_{k,n}\mm,J_{k,n}\pp$ are orthogonal idempotents for
the Eulerian peak and left peak algebras respectively.
However, we will see that $I_{k,n}\mm$ is not orthogonal to $J_{k,n}\pp$. A
first attempt to rectify this might be to calculate products such
as $I_{k,n}\mm J_{k,n}\pp$, and guess simple linear combinations
of these idempotents that are orthogonal. These product calculations
are considerably easier in the symmetric functions - that is, we calculate the images of $I_{k,n}\pp,I_{k,n}\mm,J_{k,n}\pp,J_{k,n}\mm$  
under the homomorphism $\Gamma: \Sym \rightarrow Sym$, then find simple linear combinations whose images have the 
form  $\sum_{\lambda \in \Lambda_i} \frac{1}{z_\lambda} p_\lambda$ over disjoint subsets $\Lambda_i$  of partitions of $n$ 
(see the last paragraph of Subsection~\ref{sym}).

\begin{rema}{Note that $J_{k,n}\mm$ is not a system of orthogonal
idempotents. Indeed, the proof below will show that the symmetric
function images $\Gamma(J_{k,n}\mm)$ are not of the form }$\sum_{\lambda\in\Lambda_{k}}\frac{1}{z_{\lambda}}p_{\lambda}$.
The problem is that $\{V_{2i+1}\mm\}$ is not closed under the internal
product, so the analogue of Equation~\eqref{Voddppidems}{ }cannot be
applied in Proposition~\ref{multfamilytoidempotents}. \end{rema}

\begin{theo} \label{theoidemp}
With notation as above:
\begin{enumerate}[label=\roman*)]
\item A basis of orthogonal idempotents of $\mathcal{W}_{n}$ is given
by 
\[
\frac{1}{2}(I_{k,n}\pp+I_{k,n}\mm),\quad1\leq k\leq n\text{ and }k\equiv n\mod2
\]
and
\[
\frac{1}{2}(J_{l,n}\pp+J_{l,n}\mm-I_{l,n}\pp-I_{l,n}\mm),\quad0\leq l\leq n-4\text{ and }l\equiv n\mod2.
\]
Their images under $\Gamma:\Sym\rightarrow Sym$ are \[   \sum_{\substack{ \lambda \vdash n  \\   \ell_o(\lambda)=k, \, \ell_e(\lambda)=0 }} \frac{1}{z_\lambda} p_{\lambda},   \quad 1\leq k \leq n \text{ and } k\equiv n \mod 2 \] and \[   \sum_{\substack{ \lambda \vdash n  \\   \ell_o(\lambda)=l, \, \ell_e(\lambda)>0 \text{ even }}} \frac{1}{z_\lambda} p_{\lambda},   \quad 0\leq l \leq n-4 \text{ and } l \equiv n \mod 2. \]
\item A basis of orthogonal idempotents of $\mathcal{W}_{n}^{\circ}$ is given by
\[
I_{k,n}\pp,\quad1\leq k\leq n\text{ and }k\equiv n\mod2
\]
and
\[
\frac{1}{2}(J_{l,n}\pp-I_{l,n}\pp),\quad0\leq l\leq n-2\text{ and }l\equiv n\mod2.
\]
Their images under $\Gamma:\Sym\rightarrow Sym$ are \[   \sum_{\substack{ \lambda \vdash n  \\   \ell_o(\lambda)=k, \, \ell_e(\lambda)=0 }} \frac{1}{z_\lambda} p_{\lambda},   \quad 1\leq k \leq n \text{ and } k\equiv n \mod 2 \] and \[   \sum_{\substack{ \lambda \vdash n  \\   \ell_o(\lambda)=l, \, \ell_e(\lambda)>0 }} \frac{1}{z_\lambda} p_{\lambda},   \quad 0\leq l \leq n-2 \text{ and } l \equiv n \mod 2. \]
\item A basis of orthogonal idempotents of $\mathcal{C}_{n}$ is given by
\[
\frac{1}{2}(I_{k,n}\pp+I_{k,n}\mm),\quad1\leq k\leq n\text{ and }k\equiv n\mod2
\]
and
\[
\frac{1}{2}(J_{l,n}\pp+J_{l,n}\mm-I_{l,n}\pp-I_{l,n}\mm),\quad0\leq l\leq n-4\text{ and }l\equiv n\mod2
\]
and
\[
\frac{1}{2}(J_{m,n}\pp-J_{m,n}\mm),\quad0\leq m\leq n-2\text{ and }m\equiv n\mod2.
\]
Their images under $\Gamma:\Sym\rightarrow Sym$ are \[     \sum_{\substack{ \lambda \vdash n \\ \ell_o(\lambda)=k, \; \ell_e(\lambda)=0 }} \frac{1}{z_\lambda} p_\lambda, \quad 1\leq k \leq n \text{ and } k\equiv n \mod 2 \] and \[   \sum_{\substack{ \lambda \vdash n \\ \ell_o(\lambda)=l, \; \ell_e(\lambda) \text{ even }}} \frac{1}{z_\lambda} p_\lambda, \quad 0\leq l \leq n-4 \text{ and } l\equiv n \mod 2 \] and \[   \sum_{\substack{ \lambda \vdash n \\ \ell_o(\lambda)=m, \; \ell_e(\lambda)>0 \text{ odd} }} \frac{1}{z_\lambda} p_\lambda, \quad 0\leq m \leq n-2 \text{ and } m\equiv n \mod 2. \]
\item A basis of orthogonal idempotents of $\mathcal{P}_{n}$ is given by
\[
I_{k,n}\mm,\quad1\leq k\leq n\text{ and }k\equiv n\mod2.
\]
Their images under $\Gamma:\Sym\rightarrow Sym$ are \[   \sum_{\substack{ \lambda \vdash n  \\   \ell_o(\lambda)=k, \, \ell_e(\lambda)=0 }} \frac{1}{z_\lambda} p_{\lambda},   \quad 1\leq k \leq n \text{ and } k\equiv n \mod 2. \]
\item A basis of orthogonal idempotents of $\mathcal{P}_{n}^{\circ}$ is given by
\[
J_{k,n}\pp,\quad0\leq k\leq n\text{ and }k\equiv n\mod2.
\]
Their images under $\Gamma:\Sym\rightarrow Sym$ are \[ \sum_{\substack{ \lambda \vdash n  \\   \ell_o(\lambda)=k }} \frac{1}{z_\lambda} p_{\lambda}, \quad 0\leq k \leq n \text{ and } k\equiv n \mod 2.\] 

\end{enumerate}
 \end{theo}

\begin{proof}
Observe that each set of claimed images under $\Gamma$ have cardinality
equal to the dimension of the corresponding subalgebra. Thus, if
these images are correct, $\Gamma:\Sym\rightarrow Sym$ must restrict
to an algebra isomorphism on each of these subalgebras. Hence the
claimed sets are orthogonal and idempotent if and only if their images
under $\Gamma$ are orthogonal and idempotent.

Each claimed symmetric function image is a sum of $\frac{p_{\lambda}}{z_{\lambda}}$,
the orthogonal idempotents of $Sym$, hence their sums are also idempotent.
Furthermore, for each subalgebra, no partition $\lambda$ appears
in more than one sum, so the sums are orthogonal.

So it suffices to check that the claimed idempotents are indeed in
the correct subalgebras, and have the claimed images under $\Gamma$.
Note that $I_{1}\mm$ is a series in $V_{2}\mm$, and $I_{1}\pp$
is a series in $V_{2}\pp$, so 
\begin{align*}
I_{k,n}\mm & \in\Vect(V_{2k,n}\mm)=\mathcal{P}_{n},\\
I_{k,n}\pp & \in\Vect(V_{2k,n}\pp),\\
J_{k,n}\pp & \in\Vect(V_{2k+1,n}\pp)=\mathcal{P}_{n}^{\circ},\\
J_{k,n}\mm & \in\Vect(V_{2k+1,n}\mm).
\end{align*}
By symmetry in the definitions of $I_{k}\pp,I_{k}\mm,J_{k}\pp,J_{k}\mm$,
it is clear that 
\begin{align*}
I_{k,n}\pp+I_{k,n}\mm & \in\Vect(V_{2k,n}),\\
J_{k,n}\pp+J_{k,n}\mm & \in\Vect(V_{2k+1,n}).
\end{align*}
Thus all claimed idempotents are in the correct subalgebras.

To calculate the symmetric function images, first recall that 
\begin{align*}
\Gamma(V_{1}\pp) & =\sum_{n\geq0}h_{n}=\exp_{\star}\left(\sum_{i\geq1}\frac{p_{i}}{i}\right),\\
\Gamma(V_{1}\mm) & =\sum_{n\geq0}e_{n}=\exp_{\star}\left(\sum_{i\geq1}\frac{(-1)^{i-1}p_{i}}{i}\right).
\end{align*}
Hence 
\begin{align*}
\Gamma(I_{1}\mm) & =\frac{1}{2}\Gamma(\log_{\star}(V_{1}\mm\star V_{1}\pp))\\
 & =\frac{1}{2}(\log_{\star}(\Gamma(V_{1}\mm))+\log_{\star}(\Gamma(V_{1}\pp)))\\
 & =\frac{1}{2}\left(\sum_{i\geq 1}\frac{p_{i}}{i}+\frac{(-1)^{i-1}p_{i}}{i}\right)\\
 & =\sum_{i\text{ odd }}\frac{p_{i}}{i},
\end{align*}
and similarly
\begin{align*}
\Gamma(I_{1}\pp) & =\frac{1}{2}\Gamma(\log_{\star}(V_{1}\pp\star V_{1}\mm)) \\
 & =\frac{1}{2}(\log_{\star}(\Gamma(V_{1}\pp))+\log_{\star}(\Gamma(V_{1}\mm))) \\
 & =\sum_{i\text{ odd }}\frac{p_{i}}{i}.
\end{align*}
Consequently, 
\begin{align*}
\Gamma(I_{k}\mm)=\Gamma(I_{k}\pp) & =\frac{1}{k!}\left(\sum_{i\text{ odd }}\frac{p_{i}}{i}\right)^{\star k}\\
 & =\sum_{\ell_{o}(\lambda)=k,\;\ell_{e}(\lambda)=0}\frac{1}{z_{\lambda}}p_{\lambda}.
\end{align*}

Also,
\begin{align*}
\Gamma(J_{0}\pp) & =\Gamma(V_{1}\pp)\star\Gamma(V_{1}\mm)^{\star\left(-\frac{1}{2}\right)}\star\Gamma(V_{1}\pp)^{\star\left(-\frac{1}{2}\right)}\\
 & =\Gamma(V_{1}\pp)^{\star\left(\frac{1}{2}\right)}\star\Gamma(V_{1}\mm)^{\star\left(-\frac{1}{2}\right)} \\
 & =\exp_{\star}\left(\frac{1}{2}\left(\sum_{i\geq1}\frac{p_{i}}{i}-\frac{(-1)^{i-1}p_{i}}{i}\right)\right)\\
 & =\exp_{\star}\left(\sum_{i>0 \text{ even }}\frac{p_{i}}{i}\right)=\sum_{\ell_{o}(\lambda)=0}\frac{1}{z_{\lambda}}p_{\lambda},
\end{align*}
and 
\begin{align*}
\Gamma(J_{0}\mm) & =\Gamma(V_{1}\mm)\star\Gamma(V_{1}\pp)^{\star\left(-\frac{1}{2}\right)}\star\Gamma(V_{1}\mm)^{\star\left(-\frac{1}{2}\right)}\\
 & =\Gamma(V_{1}\pp)^{\star\left(-\frac{1}{2}\right)}\star\Gamma(V_{1}\mm)^{\star\left(\frac{1}{2}\right)} \\
 & =\exp_{\star}\left(\frac{1}{2}\left(\sum_{i\geq1}-\frac{p_{i}}{i}+\frac{(-1)^{i-1}p_{i}}{i}\right)\right)\\
 & =\exp_{\star}\left(\sum_{i>0 \text{ even }}-\frac{p_{i}}{i}\right)=\sum_{\ell_{o}(\lambda)=0}\frac{(-1)^{\ell(\lambda)}}{z_{\lambda}}p_{\lambda}.
\end{align*}
So
\begin{align*}
\Gamma(J_{k}\pp) & =\Gamma(J_{0}\pp)\star\Gamma(I_{k}\pp)\\
 & =\left(\sum_{\ell_{o}(\lambda)=0}\frac{1}{z_{\lambda}}p_{\lambda}\right)\star\left(\sum_{\ell_{o}(\lambda)=k,\;\ell_{e}(\lambda)=0}\frac{1}{z_{\lambda}}p_{\lambda}\right)\\
 & =\sum_{\ell_{o}(\lambda)=k}\frac{1}{z_{\lambda}}p_{\lambda},
\end{align*}
and 
\begin{align*}
\frac{1}{2}\Gamma(J_{k}\pp+J_{k}\mm) & =\frac{1}{2}\Gamma(J_{0}\pp+J_{0}\mm)\star\Gamma(I_{k}\mm)\\
 & =\left(\sum_{\ell_{o}(\lambda)=0,\;\ell_{e}(\lambda)\text{ even }}\frac{1}{z_{\lambda}}p_{\lambda}\right)\star\left(\sum_{\ell_{o}(\lambda)=k,\;\ell_{e}(\lambda)=0}\frac{1}{z_{\lambda}}p_{\lambda}\right)\\
 & =\sum_{\ell_{o}(\lambda)=k,\;\ell_{e}(\lambda)\text{ even }}\frac{1}{z_{\lambda}}p_{\lambda},
\end{align*}
and 
\begin{align*}
\frac{1}{2}\Gamma(J_{k}\pp-J_{k}\mm) & =\frac{1}{2}\Gamma(J_{0}\pp-J_{0}\mm)\star\Gamma(I_{k}\mm)\\
 & =\left(\sum_{\ell_{o}(\lambda)=0,\;\ell_{e}(\lambda)\text{ odd }}\frac{1}{z_{\lambda}}p_{\lambda}\right)\star\left(\sum_{\ell_{o}(\lambda)=k,\;\ell_{e}(\lambda)=0}\frac{1}{z_{\lambda}}p_{\lambda}\right)\\
 & =\sum_{\ell_{o}(\lambda)=k,\;\ell_{e}(\lambda)\text{ odd }}\frac{1}{z_{\lambda}}p_{\lambda}.
\end{align*}

\end{proof}

As remarked at the end of Section~\ref{sym}, the identification of the commutative
images of the idempotents leads to analogues of the characters of
Foulkes~\cite{foulkes}, that only depend on certain features of the
cycle type:

\begin{coro} \label{foulkescharacters}
Write $\chi_{I}$ for the symmetric group character
corresponding to the ribbon skew shape for $I$. Let $\run(I),\pk(I)$
etc. denote $\run(\sigma),\pk(\sigma)$ for any permutation $\sigma$
with $\Des(\sigma)=\Des(I)$. Then
\begin{enumerate}[label=\roman*)]
\item The set of characters
\[
\left\{ \sum_{\run(I)=k}\chi_{I}:1\leq k<n\right\} 
\]
is a basis for functions $\mathfrak{S}_{n}\rightarrow\mathbb{R}$
depending only on the number of odd cycles and whether there are even
cycles, and are zero on permutations with an odd number of even cycles.
\item The set of characters
\[
\left\{ \sum_{\run^{\circ}(I)=k}\chi_{I}:1\leq k\leq n\right\} 
\]
is a basis for functions $\mathfrak{S}_{n}\rightarrow\mathbb{R}$
depending only on the number of odd cycles and whether there are even
cycles.
\item The set of characters
\begin{multline*}
\left\{ \sum_{\run(I)=2l}\chi_{I}:1\leq2l<n\right\} \amalg\left\{ \sum_{\substack{\run(I)=2k+1\\
i_{1}=1
}
}\chi_{I}:1\leq2k+1<n\right\} \\ \amalg\left\{ \sum_{\substack{\run(I)=2k+1\\
i_{1}>1
}
}\chi_{I}:1\leq2k+1<n\right\} 
\end{multline*}
is a basis for functions $\mathfrak{S}_{n}\rightarrow\mathbb{R}$
depending only on the number of odd cycles and whether the number
of even cycles is zero, even positive, or odd.
\item The set of characters
\[
\left\{ \sum_{\pk(I)=k}\chi_{I}:1\leq2k\leq n+1\right\} 
\]
is a basis for functions $\mathfrak{S}_{n}\rightarrow\mathbb{R}$
depending only on the number of odd cycles and are zero on permutations
with an odd number of even cycles.
\item The set of characters
\[
\left\{ \sum_{\pk^{\circ}(I)=k}\chi_{I}:1\leq2k+1\leq n+1\right\} 
\]
is a basis for functions $\mathfrak{S}_{n}\rightarrow\mathbb{R}$
depending only on the number of odd cycles.
\end{enumerate}
 \end{coro}
\begin{proof}
If $\left\{ \sum_{I\in\mathfrak{I}_{k}}R_{I}\right\} $ span a commutative
subalgebra of the descent algebra (for some disjoint sets $\mathfrak{I}_{k}$
of compositions), and the symmetric function image of this subalgebra
has a basis of orthogonal idempotents of the form $\left\{ \sum_{\lambda\in\Lambda_{k}}\frac{p_{\lambda}}{z_{\lambda}}\right\} $,
then $\left\{ \sum_{I\in\mathfrak{I}_{k}}\chi_{I}\right\} $ is a
basis for functions $\mathfrak{S}_{n}\rightarrow\mathbb{R}$ that
are constant on permutations with cycle type within each $\Lambda_{k}$,
and are zero on permutations with cycle type not in any $\Lambda_{k}$.
\end{proof}

\subsection{\texorpdfstring{Orthogonal idempotents for the noncommutative run algebra $\mathcal{W}_{n}^{\pm}$}
{Orthogonal idempotents for the noncommutative run algebra W±}}

We now show that the idempotents of $\mathcal{C}_{n}$ constructed
in Theorem~\ref{theoidemp} are in fact a complete set of primitive orthogonal
idempotents for $\mathcal{W}_{n}^{\pm}$.
\begin{theo}
A complete set of primitive orthogonal idempotents for $\mathcal{W}_{n}^{\pm}$ is given by
\[
\frac{1}{2}(I_{k,n}\pp+I_{k,n}\mm),\quad1\leq k\leq n\text{ and }k\equiv n\mod2
\]
and
\[
\frac{1}{2}(J_{l,n}\pp+J_{l,n}\mm-I_{l,n}\pp-I_{l,n}\mm),\quad0\leq l\leq n-4\text{ and }l\equiv n\mod2
\]
and
\[
\frac{1}{2}(J_{m,n}\pp-J_{m,n}\mm),\quad0\leq m\leq n-2\text{ and }m\equiv n\mod2.
\]
\end{theo}
\begin{proof}
The key is to show that $\mathcal{C}_{n}$ is a complement to $\mathcal{K}_{n}:=\ker\Gamma\cap\mathcal{W}_{n}^{\pm}$.
As noted in the previous section, $\Gamma$ restricted to $\mathcal{C}_{n}$
is an isomorphism, so $\mathcal{C}_{n}$ and $\mathcal{K}_{n}$ have
trivial intersection. Since $\dim\mathcal{C}_{n}=n-1+\lfloor\frac{n}{2}\rfloor$,
it must be that $\dim\mathcal{K}_{n}\leq\lceil\frac{n}{2}\rceil-1$.
Now $W_{2k,n}\pp=\omega W_{2k,n}\mm\omega$ for $1\leq k\leq\lceil\frac{n}{2}\rceil-1$,
so $W_{2k,n}\pp-W_{2k,n}\mm\in\mathcal{K}_{n}$, and these elements
are linearly independent, so $\dim\mathcal{K}_{n}=\lceil\frac{n}{2}\rceil-1$,
and $\mathcal{W}_{n}^{\pm}=\mathcal{C}_{n}\oplus\mathcal{K}_{n}$.

By \cite[Cor.~2.2]{schocker2}, the Jacobson radical of any subalgebra
$\mathcal{A}$ of $\Sym_{n}$ (under the internal product) is equal
to $\ker\Gamma\cap\mathcal{A}$. Hence $\mathcal{K}_{n}$ is the Jacobson
radical of $\mathcal{W}_{n}^{\pm}$, so any complete set of primitive
orthogonal idempotents for the complement $\mathcal{C}_{n}$ is a
complete set of primitive orthogonal idempotents for $\mathcal{W}_{n}^{\pm}$.
\end{proof}

\end{document}